\documentclass[12pt]{amsart}
\usepackage{preamble}

\title[$K$-polynomials of determinantal varieties]{Towards combinatorial derivations of $K$-polynomials for determinantal varieties}
\author{Liam Buttitta}
\address{Dept.~of Mathematics, U.~Illinois at Urbana-Champaign, Urbana, IL 61801, USA}
\email{liamb3@illinois.edu}
\author{Ada Stelzer}
\address{Dept.~of Mathematics, U.~Illinois at Urbana-Champaign, Urbana, IL 61801, USA}
\email{astelzer@illinois.edu}
\date{August 14, 2026}

\begin{document}

\begin{abstract}
    Let $\mathfrak{X}_k\subseteq\Mat_{m, n}$ denote the variety of $m\times n$ complex matrices with rank at most $k$. 
    The power series and rational expressions for the Hilbert series of $\mathfrak{X}_k$ are known by geometric arguments, and equating these expressions yields a family of formulas generalizing the classical Cauchy and dual Cauchy identities.
    We pose the problem of giving a direct combinatorial proof of these formulas for $0 < k < \min\{m, n\}$. When $k=1$ or $k=\min\{m, n\}-1$, we give such a proof via an explicit sign-reversing involution on certain sets of Littlewood--Richardson tableaux.
\end{abstract}

\maketitle

\section{Introduction}

This paper builds towards combinatorial proofs for generalizations of the \emph{Cauchy identity}, a classical result in symmetric function theory (see, e.g. \cite[Theorem 7.12.1]{ECII}).
Fix variables $\vec{x} = (x_1,\dots,x_m)$ and $\vec{y} = (y_1,\dots, y_n)$, and let $s_\lambda(\vec{x})$ denote the \emph{Schur polynomial} indexed by the partition $\lambda$ (see Section~\ref{sec:definitions} for definitions).
The Cauchy identity then states that
\begin{equation}\label{eqn:classical Cauchy}
    \sum_{\ell(\lambda) \leq \min\{m, n\}}s_\lambda(\vec{x})s_\lambda(\vec{y}) = \frac{1}{\prod_{(i, j)\in[m]\times[n]}(1-x_iy_j)}.
\end{equation}
An analogous result, the \emph{dual Cauchy identity} \cite[Theorem 7.14.3]{ECII}, states that 
\begin{equation}\label{eqn:dual Cauchy}
    \sum_{\mu\subseteq[m]\times[n]}(-1)^{|\mu|}s_\mu(\vec{x})s_{\mu'}(\vec{y})
    = \prod_{(i, j)\in[m]\times[n]}(1-x_iy_j),
\end{equation}
where $\mu'$ is the \emph{conjugate} of the partition $\mu$. 

There is a common generalization of \eqref{eqn:classical Cauchy} and \eqref{eqn:dual Cauchy}, following from work of P.~Doubilet--G.-C.~Rota--J.~Stein \cite{DRS} and A.~Lascoux \cite{Lascoux} in representation theory and commutative algebra:
\begin{theorem}[Generalized Cauchy identity]\label{thm:gen. Cauchy}
    For each $k$ such that $0\leq k\leq\min\{m, n\}$,
    \begin{equation}\label{eqn:gen. Cauchy}
        \sum_{\ell(\lambda)\leq k}s_\lambda(\vec{x})s_\lambda(\vec{y}) = \frac{\sum_{\mu\in[m-k]\times[n-k]}(-1)^{|\mu|}s_{\mu_{(k)}}(\vec{x})s_{(\mu')_{(k)}}(\vec{y})}{\prod_{(i, j)\in[m]\times[n]}(1-x_iy_j)},
    \end{equation}
    where $\mu_{(k)}$ is the \emph{$k$-stabilization} of $\mu$ \textup{(}see Definition~\ref{def:stab}\textup{)}.
\end{theorem}
Explicitly, \eqref{eqn:classical Cauchy} is the $k = \min\{m, n\}$ case of \eqref{eqn:gen. Cauchy}, whereas \eqref{eqn:dual Cauchy} is the $k = 0$ case, after multiplying across by the denominator.
Our main problem is as follows:
\begin{problem}\label{prob:main}
    Give an elementary combinatorial proof of Theorem~\ref{thm:gen. Cauchy}.
\end{problem}
In the classical cases of \eqref{eqn:classical Cauchy} and \eqref{eqn:dual Cauchy}, Problem~\ref{prob:main} is solved using the (ordinary or dual) \emph{Robinson--Schensted--Knuth correspondence} (RSK) \cite[Section 7.11]{ECII}. 
For general $k$, Problem~\ref{prob:main} remains open.

Our main new result solves Problem~\ref{prob:main} in the cases where $k = 1$ or $k = \min\{m, n\}-1$.
\begin{theorem}\label{thm:main}
    When $k = 1$ or $k = \min\{m, n\}-1$, Theorem~\ref{thm:gen. Cauchy} admits a proof by an explicit sign-reversing involution on certain sets of Littlewood--Richardson tableaux.
\end{theorem}

The proof of Theorem~\ref{thm:main} works by associating a set of pairs of Littlewood--Richardson tableaux to each partition in intervals of Young's lattice depending on $k$. 
A sign-reversing involution on these tableau-pairs for a given $k$ then solves that instance of Problem~\ref{prob:main}. 
In the cases where $k=1$ or $k=\min\{m, n\}-1$, we show that the intervals of Young's lattice considered are particularly simple: they are isomorphic to \emph{subset lattices} on the boxes of appropriate skew shapes, and moreover there is a unique pair of Littlewood--Richardson tableaux associated to each partition in the interval. 
This allows us to give explicit sign-reversing involutions in these cases.

\subsection{History and motivation}
We briefly summarize the origins of Theorem~\ref{thm:gen. Cauchy} and motivate our search for an elementary combinatorial proof.
Let $U$ and $W$ be complex vector spaces of dimensions $m$ and $n$ respectively, and identify the tensor product $U\boxtimes W$ with the affine space $\Mat_{m, n}$ of complex $m\times n$ matrices. 
The \emph{coordinate ring} $\C[\Mat_{m, n}]$ is then the symmetric algebra $S((U\boxtimes W)^*)$, a polynomial ring in an $m\times n$ matrix $Z = [z_{ij}]$ of variables. 
This ring is a polynomial representation of $G = GL(U)\times GL(W)$ via the left action
\begin{equation}
    (g, h)\cdot f(Z) \coloneq f(gZh^t).
\end{equation}
The maximal torus $T\subseteq G$ of pairs of invertible diagonal matrices induces a $\Z^{m+n}$-multigrading on $\C[\Mat_{m, n}]$ in the sense of \cite[Definition 8.1]{Miller.Sturmfels}; explicitly, each variable $z_{ij}$ is assigned the monomial weight $x_iy_j$. 
The \emph{character} of $\C[\Mat_{m, n}]$ as a $G$-representation is exactly its \emph{Hilbert series} with respect to this multigrading. 
Computing this Hilbert series as the weight generating series for monomials in the variables $z_{ij}$ yields the right side of the Cauchy identity \eqref{eqn:classical Cauchy}. 
The left side of \eqref{eqn:classical Cauchy} then expresses the character of $\C[\Mat_{m, n}]$ in terms of irreducible $G$-characters (we call this expression the \emph{$G$-equivariant Hilbert series} of $\C[\Mat_{m, n}]$).
In this sense, the RSK correspondence previously mentioned is a weight-preserving bijection deriving the $G$-irreducible multiplicities of $\C[\Mat_{m, n}]$ combinatorially.

Theorem~\ref{thm:gen. Cauchy} has a geometric interpretation generalizing that of \eqref{eqn:classical Cauchy}. 
The only $G$-stable subvarieties of $\Mat_{m, n}$ are the \emph{classical determinantal varieties}
\[\mathfrak{X}_k = \{M\in\Mat_{m, n}:\mathrm{rank}(M)\leq k\}.\]
The left side of \eqref{eqn:gen. Cauchy} is precisely the $G$-equivariant Hilbert series of the coordinate ring $\C[\mathfrak{X}_k]$, first computed by P.~Doubilet--G.-C.~Rota--J.~Stein in their work \cite{DRS} proving the first fundamental theorem of invariant theory over arbitrary commutative rings. 

Now, by the classical \emph{Hilbert--Serre theorem} (see, e.g., \cite[Theorem 8.20]{Miller.Sturmfels}), for any $T$-stable embedded subvariety $\mathfrak{X}\subseteq\Mat_{m, n}$, there exists an integer polynomial $K_\mathfrak{X}(\vec{x},\vec{y})$ such that
\[HS_\mathfrak{X}(\vec{x},\vec{y}) = \frac{K_\mathfrak{X}(\vec{x},\vec{y})}{\prod_{i, j}(1-x_iy_j)}.\footnote{The result in \cite{Miller.Sturmfels} applies in a more general context where $K_\mathfrak{X}(\vec{x},\vec{y})$ may be a Laurent polynomial, but our explicit choice of group action on ${\sf Mat}_{m, n}$ ensures that the $K$-polynomial is an honest polynomial here.}\]
This \emph{$K$-polynomial} $K_\mathfrak{X}(\vec{x},\vec{y})$ is the Euler characteristic of the minimal free resolution of $\C[\mathfrak{X}]$. 
A.~Lascoux \cite{Lascoux} constructed $G$-equivariant minimal free resolutions for each classical determinantal variety $\mathfrak{X}_k$, from which one can read off their $G$-equivariant $K$-polynomials and derive the right side of \eqref{eqn:gen. Cauchy}. 
In particular, when $k = 0$, Lascoux's construction specializes to the \emph{Koszul complex} resolving the base field $\C\subset\C[\Mat_{m, n}]$ as a $\C[\Mat_{m, n}]$-module, from which one obtains the dual Cauchy identity \eqref{eqn:dual Cauchy}. 
For a detailed exposition of these results, see J.~Weyman's book \cite[Section 6.1]{Weyman}.

Our motivation for considering Problem~\ref{prob:main} stems from the second author's work with A.~Price and A.~Yong on equivariant homological invariants of \emph{matrix Schubert varieties} \cite{AAA1}. 
In \cite[Appendix A]{AAA1}, it was noted that combinatorial understanding of the equivariant Hilbert series of these varieties is more developed than that of their equivariant $K$-polynomials. 
In particular, Lascoux's methods have not been used successfully to compute equivariant $K$-polynomials of matrix Schubert varieties, although their equivariant Hilbert series are computed by \cite[Theorems~1.11 and 1.14]{AAA1}. 
By exploring combinatorial proofs for Theorem~\ref{thm:gen. Cauchy}, we aim to discover methods that may apply to matrix Schubert varieties and other determinantal varieties for which the minimal free resolutions remain unknown. 

\subsection{Organization}

The organization of this paper is as follows. Section~\ref{sec:definitions} gathers standard definitions and results used in our arguments. Section~\ref{sec:tech} contains proofs of some technical lemmas that we frequently refer back to in the proof of Theorem~\ref{thm:main}.

In Section~\ref{sec:dual Cauchy}, we use the dual Cauchy identity \eqref{eqn:dual Cauchy} and the Littlewood--Richardson rule (Theorem \ref{thm:LR rule}) to convert the algebraic statement of \eqref{eqn:gen. Cauchy} into the two-part combinatorial Problem~\ref{prob:dual}. We solve the first part of this problem for all $k$ in Theorem~\ref{thm:dual coeff. = 1 if stab.} and the second part when $k=1$ in Theorem~\ref{thm:k = 1 solution}, proving the $k=1$ case of Theorem~\ref{thm:main}. 

Similarly, in Section~\ref{sec:classical Cauchy}, we use the Cauchy identity \eqref{eqn:classical Cauchy} and the Littlewood--Richardson rule to convert the algebraic statement of \eqref{eqn:gen. Cauchy} into the two-part combinatorial Problem~\ref{prob:classical}. We solve the first part of this problem for all $k$ in Proposition~\ref{prop:classical coeffs. = 1 if equal} and the second part when $k=\min\{m, n\}-1$ in Theorem~\ref{thm:classical involution}, proving the $k=\min\{m, n\}-1$ case of Theorem~\ref{thm:main}. 

The concluding Section~\ref{sec:further} provides commentary on the general case of Problem~\ref{prob:main}.

\section{Basic definitions}
\label{sec:definitions}

\subsection{Partitions and tableaux}
\label{subsec:partitions}
\begin{definition}
    A \emph{partition} $\lambda$ of size $n\in\Z_{\geq 0}$ (denoted $\lambda \vdash n$) is a tuple of nonnegative integers $(\lambda_1, \ldots, \lambda_k)$ such that $\lambda_1 \geq \cdots \geq \lambda_k$ and $\lambda_1 + \cdots + \lambda_k = n$.
\end{definition}
We identify a partition $\lambda$ with its \emph{Young diagram} in English notation: a table of boxes in which the $i$-th row from the top contains $\lambda_i$ boxes. The \emph{length} of $\lambda$, denoted $\ell(\lambda)$, is the number of (nonzero) rows in its diagram. The diagram of the \emph{conjugate} partition $\lambda'$ is obtained by transposing the rows and columns of $\lambda$'s diagram. 
We refer to boxes in Young diagrams with matrix coordinates, so the box in row $i$ and column $j$ of $\lambda$ has coordinates $(i, j)$.
\begin{definition}
    Let $\lambda$ be a partition and let $\mathsf b$ be a box in its Young diagram.
    \begin{itemize}
        \item
            We call $\mathsf b$ an \emph{edge box} if it does not have both a box below it and a box to its right.
        \item
            We call $\mathsf b$ a \emph{corner box} if it has neither a box below it nor a box to its right.
            (All corner boxes are also edge boxes.)
    \end{itemize}
\end{definition}
\begin{example}
    The partition $\lambda = (4, 2, 1) \vdash 7$ has Young diagram
    
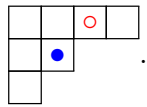
\[
        \begin{yctableau}
            \ & \ & \red{\circ} & \ \\
            \ & \blue{\bullet} \\
            \
        \end{yctableau}.
    \]
    The box $(1, 3)$ marked with a red circle is an edge box, whereas the box $(2, 2)$ marked with a blue bullet is a corner box. The conjugate partition is $\lambda' = (3, 2, 1, 1)$. 
\end{example}
\begin{definition}
    The \emph{rank} of a partition $\lambda$, denoted $r(\lambda)$, is the greatest integer $k$ such that $\lambda_k \geq k$. Equivalently, $r(\lambda)$ is the side length of the largest square (known as the \emph{Durfee square}) contained in the diagram of $\lambda$.
\end{definition}
The following operation on partitions appears in the statement of \eqref{eqn:gen. Cauchy}.
\begin{definition}\label{def:stab}
    The \emph{$k$-stabilization} $\lambda_{(k)}$ of a partition $\lambda$ is the partition obtained by adding $k$ rows of length $r(\lambda)$ to $\lambda$.
        For example, the 1-stabilization of $\lambda = (4, 3, 1)$ is $\lambda_{(1)} = (4, 3, 2, 1)$:
        \[
            \newcommand{\dmark}{\green{\star}}
            \newcommand{\smark}{\red{\circ}}
            \lambda = \begin{yctableau}
                \dmark & \dmark & \ & \ \\
                \dmark & \dmark & \ \\
                \
            \end{yctableau}
            \quad\Rightarrow\quad
            \lambda_{(1)} = \begin{yctableau}
                \dmark & \dmark & \ & \ \\
                \dmark & \dmark & \ \\
                \smark & \smark \\
                \
            \end{yctableau},
            ~\lambda_{(2)} = \begin{yctableau}
                \dmark & \dmark & \ & \ \\
                \dmark & \dmark & \ \\
                \smark & \smark \\
                \smark & \smark \\
                \
            \end{yctableau},
            ~\cdots.
        \]
        The Durfee square of $\lambda$ is marked with green stars, and the rows added by stabilization are marked with red circles.
\end{definition}

Given partitions $\lambda$ and $\mu$, we say that $\lambda$ \emph{contains} $\mu$ ($\lambda\supseteq\mu$) if the diagram for $\lambda$ contains the diagram for $\mu$, that is, if $\ell(\lambda)\geq\ell(\mu)$ and $\lambda_i\geq\mu_i$ for each $i\in[\ell(\mu)]$. 
This partial order defines \emph{Young's lattice} on the set of all partitions. 
Meets are computed by intersecting diagrams (i.e., $(\alpha\cap\beta)_i = \min\{\alpha_i,\beta_i\}$), and joins by taking unions (i.e., $(\alpha\cup\beta)_i = \max\{\alpha_i,\beta_i\}$).
Most of our arguments concern \emph{tableaux}, fillings of \emph{skew} Young diagrams.
\begin{definition}
    A \emph{skew shape} $\lambda / \mu$ is a pair of partitions $\lambda, \mu$ such that $\lambda \supseteq \mu$.
    A box $\mathsf b$ is contained in the skew shape $\lambda / \mu$ if and only if $\mathsf b \in \lambda$ and $\mathsf b \notin \mu$.
\end{definition}
\begin{example}
    Let $\lambda = (4, 2, 1), \mu = (2)$.
    Then the skew shape $\lambda / \mu$ has Young diagram
    \[
        \begin{yctableau}
            \none & \none & \ & \ \\
            \ & \ \\
            \
        \end{yctableau}.
    \]
\end{example}
We identify each partition $\lambda$ with the skew shape $\lambda / \emptyset$, where $\emptyset$ is the empty partition.
\begin{definition}
    Let $\lambda / \mu$ be a skew shape.
    A \emph{tableau} $T$ of shape $\lambda / \mu$ is a labeling of each box in $\lambda / \mu$ with a positive integer.
\end{definition}
\begin{definition}
    Let $T$ be a tableau and let $n$ be the largest integer appearing as a label in $T$.
    The \emph{content} of $T$ is the tuple $c(T) = (c_1, \ldots, c_n)$, where $c_i$ is the number of boxes labeled $i$ in $T$.
    The \emph{weight} $\vec x^T$ of $T$ is the monomial in $x_1, \ldots, x_n$ whose exponent vector is $c(T)$.
\end{definition}
\begin{definition}
    A tableau $T$ is \emph{semistandard} if its labels are weakly increasing along rows from left to right, and strictly increasing along columns from top to bottom.
    Let $\SSYT(\lambda, n)$ denote the set of all semistandard tableaux of shape $\lambda$ with labels from $[n]$.
\end{definition}
\begin{example}
    The tableau
    \[
        T = \begin{yctableau}
            1 & 1 & 2 & 3 \\
            2 & 4 \\
            4
        \end{yctableau}
    \]
    is semistandard of shape $\lambda = (4, 2, 1)$, content $(2, 2, 1, 2)$, and weight $x_1^2x_2^2x_3x_4^2$.
\end{example}
\begin{definition}
    Let $\lambda$ be a partition and let $\vec x = (x_1, \ldots, x_n)$ be a set of variables.
    The \emph{Schur polynomial} $s_\lambda(\vec x)$ is the weight generating series
    \[
        s_\lambda(\vec x) = \sum_{T \in \SSYT(\lambda, n)} \vec x^T.
    \]
\end{definition}
\begin{remark}
    \label{rmk:Schur vanishing}
    Semistandard Young tableaux are strictly increasing along columns, and so the label of the first box in row $i$ of a tableau must be at least $i$.
    Therefore, the Schur polynomial $s_\lambda(x_1, \ldots, x_n)$ in $n$ variables vanishes if $\ell(\lambda) > n$.
\end{remark}

\subsection{The Littlewood--Richardson rule}
\label{subsec:LR rule}
The \emph{Littlewood--Richardson rule} expands products of Schur polynomials in the linear basis of Schur polynomials combinatorially.
\begin{definition}
    A word $w$ on the alphabet $[n]$ is \emph{lattice} if for all $i\in[n-1]$, each initial segment of $w$ contains at least as many $i$'s as $(i+1)$'s.
\end{definition}
\begin{definition}
    Let $T$ be a tableau. The \emph{reverse reading word} $\revword(T)$ is formed by reading the entries of $T$ along the rows from right to left and from top to bottom.
    $T$ is a \emph{Littlewood--Richardson tableau} if it is semistandard and $\revword(T)$ is lattice.
\end{definition}
\begin{example}
    Let $\nu = (4, 3, 3, 1)$, $\lambda = (2, 1)$, and $\mu = (3, 3, 2)$. Consider the tableaux
    \[
        T_1 = \begin{yctableau}
            \none & \none & 1 & 1 \\
            \none & 1 & 2 \\
            2 & 2 & 3 \\
            3
        \end{yctableau}
        \quad\text{and}\quad
        T_2 = \begin{yctableau}
            \none & \none & 1 & 2 \\
            \none & 1 & 2 \\
            1 & 3 & 3 \\
            2
        \end{yctableau}.
    \]
    Both are semistandard tableaux of shape $\nu / \lambda$ and content $\mu$.
    However, $\revword(T_2) = 21213312$, which fails the lattice condition for $i = 1$, so only $T_1$ is Littlewood--Richardson.
\end{example}
\begin{definition}
    Let $\LR(\nu / \lambda, \mu)$ denote the set of all Littlewood--Richardson tableaux of shape $\nu / \lambda$ and content $\mu$.
    The \emph{Littlewood--Richardson coefficient} $c^\nu_{\lambda, \mu}$ enumerates ${\sf LR}(\nu/\lambda,\mu)$:
    \[
        c^\nu_{\lambda, \mu} = |\LR(\nu / \lambda, \mu)|.
    \]
\end{definition}
\begin{theorem}[{Littlewood--Richardson rule, \cite[Theorem A1.3.3]{ECII}}]
    \label{thm:LR rule}
    Let $\vec{x} = (x_1,\dots, x_n)$ and let $\lambda$ and $\mu$ be two partitions.
    Then
    \begin{equation}
        \label{eqn:LR rule}
        s_\lambda(\vec x)s_\mu(\vec x) = \sum_{\nu} c_{\lambda, \mu}^{\nu} s_\nu(\vec x).
    \end{equation}
\end{theorem}

\subsection{The Pieri rules}
\label{subsec:Pieri rules}
The \emph{Pieri rule} and its relatives are special cases of the Littlewood--Richardson rule where all coefficients that appear are either zero or one.
\begin{definition}
    A \emph{row} (resp. \emph{column}) \emph{partition} is a partition with at most one row (resp. one column).
\end{definition}
\begin{definition}[Horizontal strip]
    A \emph{horizontal} (resp. \emph{vertical}) \emph{strip} is a skew shape with at most one box in each column (resp. each row).
\end{definition}
\begin{theorem}[{Pieri rule, \cite[Equation 2.2(4)]{Fulton}}]
    \label{thm:Pieri rule}
    Let $\nu / \lambda$ be a skew shape and let $(s)$ denote the row partition of size $s$.
    Then
    \[
        c^\nu_{\lambda, (s)} = \begin{cases}
            1 & \text{if $\nu / \lambda$ is a horizontal strip of size $s$,} \\
            0 & \text{otherwise}.
        \end{cases}
    \]
\end{theorem}
\begin{theorem}[{Dual Pieri rule, \cite[Equation 2.2(5)]{Fulton}}]
    \label{thm:dual Pieri rule}
    Let $\nu / \lambda$ be a skew shape and let $(1^s)$ denote the column partition of size $s$.
    Then
    \[
        c^\nu_{\lambda, (1^s)} = \begin{cases}
            1 & \text{if $\nu / \lambda$ is a vertical strip of size $s$}, \\
            0 & \text{otherwise}.
        \end{cases}
    \]
\end{theorem}
The \emph{hook Pieri rule} below is known but non-standard, so we provide a full proof.
\begin{definition}
    \label{def:hook partition}
    The \textit{hook partition} with width $w \in \Z_{> 0}$ and height $h \in \Z_{> 0}$ is the partition 
    \[H(w, h) \coloneq (w, \underbrace{1, \ldots, 1}_{h-1}) = (w) \cup (1^h).\]
\end{definition}
\begin{definition}
    \label{def:trimming}
    Let $\lambda$ be a partition.
    Then the \emph{row trimming} $\rtrim_k(\lambda)$ is the partition obtained by removing the rightmost  $k$ boxes from each row of $\lambda$.
    Similarly, the \emph{column trimming} $\ctrim_k(\lambda)$ is obtained by removing the bottommost $k$ boxes from each column of $\lambda$.
\end{definition}
\begin{lemma}
    \label{lem:LR distinct rows}
    Let $T$ be a Littlewood--Richardson tableau containing at least two distinct labels $i < j$.
    Then the occurrence of $i$ first in $\revword(T)$ must come from a row strictly above the row containing the occurrence of $j$ first in $\revword(T)$.
\end{lemma}
\begin{proof}
    The condition that $\revword(T)$ is lattice implies that the first occurrence of $i$ is neither below nor directly to the left of the first occurrence of $j$, and the condition that the rows of $T$ are weakly increasing implies that no occurrence of $i$ is directly right of any occurrence of $j$.
    Therefore, the first occurrence of $i$ must be above the first occurrence of $j$.
\end{proof}
\begin{theorem}[Hook Pieri rule]
    \label{thm:hook Pieri rule}
    Let $w, h \in \Z_{> 0}$.
    Let $\nu / \lambda$ be a skew shape such that $\ell(\nu) \leq h$.
    Then
    \[
        c^\nu_{\lambda, H(w, h)} = \begin{cases}
            1 & \text{if $\ell(\nu) = h$, $\lambda \subseteq \rtrim_1(\nu)$, and $\rtrim_1(\nu) / \lambda$ is a horizontal strip}, \\
            0 & \text{otherwise}.
        \end{cases}
    \]
\end{theorem}
\begin{proof}
    Suppose that there exists a Littlewood--Richardson tableau $T \in \LR(\nu / \lambda, H(w, h))$.
    Then each of the labels $2, \ldots, h$ occur exactly once in $T$, whereas the label $1$ occurs $w$ times.
    We show that $T$ is uniquely determined by $\nu$ and $\lambda$, and that the existence of $T$ forces the conditions that $\ell(\nu) = h$, $\lambda\subset\rtrim_1(\nu)$, and $\rtrim(\nu)/\lambda$ is a horizontal strip.

    Since $\ell(\nu)\leq h$ by assumption, Lemma~\ref{lem:LR distinct rows} implies that the first occurrence of each label $i$ in $T$ must be in row $i$. 
    Thus $\ell(\nu) = h$.
    The weakly increasing condition on the rows of $T$ then implies that the rightmost entry of row $i$ of $T$ is an $i$ for all $i\in[h]$.
    Since $T$ has shape $\nu/\lambda$, we deduce that $\lambda\subseteq\rtrim_1(\nu)$.
    The remaining boxes of $T$ form the skew shape $\rtrim_1(\nu)/\lambda$ and are filled with the $w-1$ remaining $1$ labels. 
    Since the columns of $T$ are strictly increasing, it follows that $\rtrim_1(\nu)/\lambda$ is a horizontal strip.
    Moreover, we see that $T$ is uniquely determined by the shapes $\nu$ and $\lambda$ as claimed. \qedhere
\end{proof}

\section{Technical lemmas}\label{sec:tech}
In this section, we collect some technical lemmas about \emph{bound partitions}, a construction which we use frequently in proving Theorem \ref{thm:main}.

Throughout this section, let $\lambda, \mu$ be two partitions and let $k \in \Z_{\geq 0}$.
\begin{definition}
    \label{def:bound partitions}
    The \emph{upper bound partition} of $\lambda$ and $\mu$ is $\nu_+(\lambda, \mu)\coloneq\lambda \cap \mu$, and the \emph{$k$-th lower bound partition} is $\nu_-^k(\lambda, \mu)\coloneq \ctrim_k(\lambda) \cup \rtrim_k(\mu)$ (where $\ctrim$ and $\rtrim$ are as in Definition \ref{def:trimming}).
\end{definition}
\begin{definition}
    Let $\nu_+ = \nu_+(\lambda, \mu)$ and let $\mathsf b = (i, j)$ be an edge box in $\nu_+$.
    We say that ${\sf b}$ is \emph{shadowed downwards} if $\lambda / \nu_+$ contains a box in position $(i+1, j)$, and \emph{shadowed rightwards} if $\mu / \nu_+$ contains a box in  position $(i, j+1)$.
\end{definition}
\begin{lemma}
    \label{lem:everything is shadowed}
    Let $\nu_+ = \nu_+(\lambda, \mu)$ and let $\nu_-^k = \nu_-^k(\lambda, \mu)$.
    Let $\mathsf b = (i, j)$ be a box in $\nu_+$.
    
    If $\nu_-^k = \nu_+$, then the following statements hold:
    \begin{enumerate}
        \item there exist at most $k$ boxes below ${\sf b}$ in column $j$ of $\lambda / \nu_+$ and at most $k$ boxes right of ${\sf b}$ in row $i$ of $\mu / \nu_+$; and
        
        \item there exist exactly $k$ boxes below $\mathsf b$ in column $j$ of $\lambda / \nu_+$, or else there exist exactly $k$ boxes right of $\mathsf b$ in row $i$ of $\mu / \nu_+$.
    \end{enumerate}
\end{lemma}
\begin{proof}
    Since $\ctrim_k(\lambda) \subseteq \nu_-^k = \nu_+$, we know that $\lambda / \nu_+ \subseteq \lambda / \ctrim_k(\lambda)$.
    And by construction $\lambda / \ctrim_k(\lambda)$ contains at most $k$ boxes in each column, so $\lambda / \nu_+$ contains at most $k$ boxes in each column as well.
    Similarly, $\mu / \nu_+$ contains at most $k$ boxes in each row. This proves part (1).

    Now, if $\nu_-^k = \nu_+$, then by definition $\ctrim_k(\lambda) \cup \rtrim_k(\mu) = \lambda \cap \mu$.
    Therefore, if $\mathsf b \in \nu_+$, then either $\mathsf b \in \ctrim_k(\lambda)$, that is, there exist at least $k$ boxes below $\mathsf b$ in column $j$ of $\lambda / \nu_+$; or else $\mathsf b \in \rtrim_k(\lambda)$, that is, there exist at least $k$ boxes right of $\mathsf b$ in row $i$ of $\mu / \nu_+$.
    Combining this fact with part (1) proves part (2).
\end{proof}
\begin{lemma}
    \label{lem:* if nu_- = nu_+}
    Let $\lambda, \mu$ be two partitions.
    Let $\nu_+ = \nu_+(\lambda, \mu)$ and let $\nu_-^k = \nu_-^k(\lambda, \mu)$.
    Let $C = \{\mathsf c_1 = (i_1, j_1), \ldots, \mathsf c_q = (i_q, j_q)\}$ be the set of corner boxes in the Young diagram of $\nu_+$, ordered from bottom-left to top-right \textup{(}so $i_1 > \cdots > i_q$ and $j_1 < \cdots < j_q$\textup{)}.

    If $\nu_-^k = \nu_+$, then there exist $u, v\in[q]$ such that:
    \begin{enumerate}
    \item
        each corner box ${\sf c}_t$ is shadowed downwards if and only if $t\leq u$;
    \item
        each corner box ${\sf c}_t$ is shadowed rightwards if and only if $t \geq v$; and
    \item
        either $u = v$ or $u = v - 1$.
    \end{enumerate}
    Moreover, the following statements hold:
    \begin{enumerate}
    \item[(4)]
        if an edge box $\mathsf b = (i, j)$ is shadowed downwards, and the box $\mathsf b' = (i, j - 1)$ is an edge box, then $\mathsf b'$ is shadowed downwards; and
    \item[(5)]
        if an edge box $\mathsf b = (i, j)$ is shadowed rightwards, and the box $\mathsf b' = (i - 1, j)$ is an edge box, then $\mathsf b'$ is shadowed rightwards.
    \end{enumerate}
\end{lemma}
\begin{proof}
    Consider two consecutive corner boxes $\mathsf c_t, \mathsf c_{t + 1} \in C$.
    Since these boxes are consecutive corners, we must have that the box $\mathsf d_t = (i_{t + 1} + 1, j_t + 1) \notin \nu_+$.
    Note that $i_{t + 1} + 1 \leq i_t$, so if the box $\mathsf a_t = (i_t, j_t + 1) \in \mu$, then $\mathsf d_t \in \mu$.
    Likewise, $j_t + 1 \leq j_{t + 1}$, so if the box $\mathsf b_t = (i_{t + 1} + 1, j_{t + 1}) \in \lambda$, then $\mathsf d_t \in \lambda$.
    Hence if $\mathsf a_t \in \mu$, then we must have that $\mathsf b_t \notin \lambda$.
    Equivalently, if $\mathsf c_t$ is shadowed rightwards, then $\mathsf c_{t + 1}$ is not shadowed downwards.
    \begin{figure}[h]
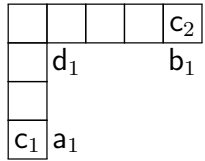

        \[
            \ytableausetup{boxsize=1.2em}
            \begin{yctableau}
                \ & \ & \ & \ & \mathsf c_2 \\
                \ & \none[\mathsf d_1] & \none & \none & \none[\mathsf b_1] \\
                \ \\
                \mathsf c_1 & \none[\mathsf a_1]
            \end{yctableau}
        \]
        \caption{An example of the setup for the proof of Lemma \ref{lem:* if nu_- = nu_+}, letting $t = 1$}
    \end{figure}

    But by Lemma \ref{lem:everything is shadowed}(2), every corner box in $\nu_+$ is either shadowed rightwards or shadowed downwards.
    Hence statements (1--3) hold.
    Furthermore, if statement (4) did not hold, then $\lambda$ would not be a valid partition shape; and likewise, if statement (5) did not hold, then $\mu$ would not be a valid partition shape.
\end{proof}

Under the further assumption that $k=1$, we can completely describe $\lambda/\nu_+$ and $\mu/\nu_+$.
\begin{lemma}
    \label{lem:stab. shapes if nu_- = nu_+}
    Let $\lambda$, $\mu$, $C$, $u$, and $v$ be as in Lemma \ref{lem:* if nu_- = nu_+}.
    Let $\nu_+ = \nu_+(\lambda, \mu)$ and let $\nu_-^1 = \nu_-^1(\lambda, \mu)$.
    If $\nu_-^1 = \nu_+$, then the following statements hold:
    \begin{enumerate}
    \item
        $\lambda / \nu_+$ consists of exactly one box in each of the first $j_u$ columns; and
    \item
        $\mu / \nu_+$ consists of exactly one box in each of the first $i_v$ rows.
    \end{enumerate}
\end{lemma}
\begin{proof}
    Lemma~\ref{lem:* if nu_- = nu_+}(1, 4) implies that $\lambda / \nu_+$ has at least one box in each of the first $j_u$ columns and none in columns $j > j_u$.
    Lemma \ref{lem:everything is shadowed}(1) then implies that $\lambda / \nu_+$ has exactly one box in each of the first $j_u$ columns.
    
    Likewise, Lemma \ref{lem:* if nu_- = nu_+}(2, 5) implies that $\mu / \nu_+$ has at least one box in each of the first $i_v$ rows and none in rows $i > i_v$.
    Lemma \ref{lem:everything is shadowed}(1) then implies that $\mu / \nu_+$ has exactly one box in each of the first $i_v$ rows.
\end{proof}

We conclude this section with a lemma showing that certain intervals in Young's lattice are subset lattices.
These intervals feature prominently in the proof of our main theorems.
\begin{lemma}
    \label{lem:nu corners}
    Let $\nu_+ = \nu_+(\lambda, \mu)$ and let $\nu_-^1 = \nu_-^1(\lambda, \mu)$.
    If $\nu_-^1 \subseteq \nu_+$, then each box $\mathsf b \in \nu_+ / \nu_-^1$ is a corner box of $\nu_+$. Thus the interval $[\nu_-^1, \nu_+]$ is a subset lattice on the boxes of $\nu_+/\nu_-^1$.
\end{lemma}
\begin{proof}
    Since $\mathsf b \notin \nu_-^1$, the definition of $\nu_-^1$ implies that $\mathsf b$ is the bottom box in its column in $\lambda$ and the rightmost box in its row in $\mu$.
    In other words, $\mathsf b$ has neither a box immediately below it in $\lambda$ nor a box immediately to the right of it in $\mu$.
    But $\nu_+ = \lambda \cap \mu$; hence $\mathsf b$ has neither a box immediately below it nor a box immediately to the right of it in $\nu_+$.
\end{proof}

\section{Proof of the \texorpdfstring{$k=1$}{superminimal} case}
\label{sec:dual Cauchy}

The goal of this section is to solve the $k=1$ case of Problem~\ref{prob:main}. 
We do this by first using the dual Cauchy identity \eqref{eqn:dual Cauchy} to derive Problem~\ref{prob:dual}, a combinatorial reformulation of the generalized Cauchy identity \eqref{eqn:gen. Cauchy}.

\begin{definition}
    Let $\alpha, \beta$ be a pair of partitions such that $|\alpha| = |\beta|$.
    Then we define
    \[
        X_{\alpha, \beta}(\lambda, \nu)
        = \LR(\alpha / \nu, \lambda) \times \LR(\beta / \nu', \lambda),
        \quad
        X_{\alpha, \beta, k}
        = \bigsqcup_{\substack{\lambda, \nu \\ \ell(\lambda) \leq k}} X_{\alpha, \beta}(\lambda, \nu).
    \]
    Let the \emph{sign} of $\chi \in X_{\alpha, \beta}(\lambda, \nu)$ be $\sgn(\chi) = (-1)^{|\nu|}$.
    Then we define
    \begin{equation}
        \begin{split}
            C_{\alpha, \beta}(\lambda, \nu)
            &= |X_{\alpha, \beta}(\lambda, \nu)|
            = c^\alpha_{\nu, \lambda}c^\beta_{\nu', \lambda}, \\
            C_{\alpha, \beta, k}
            &= \sum_{\chi \in X_{\alpha, \beta, k}} \sgn(\chi)
            = \sum_{\substack{\lambda, \nu \\ \ell(\lambda) \leq k}} (-1)^{|\nu|} C_{\alpha, \beta}(\lambda, \nu).
        \end{split}
    \end{equation}
\end{definition}
\begin{theorem}
    \label{thm:dual Cauchy drv.}
    The generalized Cauchy identity \eqref{eqn:gen. Cauchy} is equivalent to the statement that
    \begin{equation}
        \label{eqn:dual coeffs.}
        C_{\alpha, \beta, k} = \begin{cases}
            (-1)^{|\mu|} & \text{if there exists $\mu$ such that $\alpha = \mu_{(k)}, \beta = (\mu')_{(k)}$}, \\
            0 & \text{otherwise}.
        \end{cases}
    \end{equation}
\end{theorem}
\begin{proof}
    Multiply across in \eqref{eqn:gen. Cauchy} to obtain that
    \[
        \bigg[\prod_{(i, j) \in [m] \times [n]} (1 - x_i y_j)\bigg]
        \bigg[\sum_{\ell(\lambda) \leq k} s_\lambda(\vec x)s_\lambda(\vec y)\bigg]    
        = \sum_{\mu \in [m - k] \times [n - k]}
            (-1)^{|\mu|} s_{\mu_{(k)}}(\vec x)s_{(\mu')_{(k)}}(\vec y).
    \]
    Then apply the dual Cauchy identity \eqref{eqn:dual Cauchy} to rewrite the product term on the left-hand side:
    \[
        \bigg[\sum_\nu (-1)^{|\nu|} s_\nu(\vec x)s_{\nu'}(\vec y)\bigg]
        \bigg[\sum_{\ell(\lambda) \leq k} s_\lambda(\vec x)s_\lambda(\vec y)\bigg] 
        = \sum_{\mu \in [m - k] \times [n - k]}
            (-1)^{|\mu|} s_{\mu_{(k)}}(\vec x)s_{(\mu')_{(k)}}(\vec y),
    \]
    Apply the Littlewood--Richardson rule \eqref{eqn:LR rule} to multiply the two pairs of Schur polynomials on the left-hand side:
    \[
        \sum_{\alpha, \beta}\Bigg[
            \bigg(
                \sum_{\substack{\nu, \lambda \\ \ell(\lambda) \leq k}}
                (-1)^{|\nu|}c^\alpha_{\nu, \lambda}c^\beta_{\nu',\lambda}
            \bigg)s_\alpha(\vec x)s_{\beta}(\vec y)
        \Bigg]
        = \sum_{\mu \in [m - k] \times [n - k]}
            (-1)^{|\mu|} s_{\mu_{(k)}}(\vec x)s_{(\mu')_{(k)}}(\vec y).
        \qedhere
    \]
    The Schur polynomials form a linear basis of the space of symmetric functions \cite[Corollary 7.10.6]{ECII}, so the coefficients of any two decompositions must be equal.
\end{proof}
By Theorem \ref{thm:dual Cauchy drv.}, to solve Problem \ref{prob:main} it suffices to solve the following problem:
\begin{problem}
    \label{prob:dual}
    Let $\alpha, \beta$ be a pair of partitions such that $|\alpha| = |\beta|$ and let $k \in \Z_{\geq 0}$.
    \begin{enumerate}
    \item
        Prove that if $\alpha = \mu_{(k)}$ and $\beta = (\mu')_{(k)}$ for some partition $\mu$, then $X_{\alpha, \beta, k}$ contains exactly one element, which has sign $(-1)^{|\mu|}$.
    \item
        If no such $\mu$ exists, construct a sign-reversing involution on $X_{\alpha, \beta, k}$.
    \end{enumerate}
\end{problem}
The following subsections solve Problem~\ref{prob:dual}(1) for all $k$ and Problem~\ref{prob:dual}(2) when $k=1$.

\subsection{The non-involution case}
This subsection proves Theorem~\ref{thm:dual coeff. = 1 if stab.}, which solves Problem \ref{prob:dual}(1) for all $k$.

Throughout this subsection, let $\alpha, \beta$ be a pair of partitions such that $|\alpha| = |\beta|$ and let $k \in \Z_{\geq 0}$.
Let $\nu_+ = \nu_+(\alpha, \beta')$ (note the conjugate) and let $\nu_-^k = \nu_-^k(\alpha, \beta')$.
\begin{prop}
    \label{prop:nu interval}
    If $\nu^k_-\not\subseteq\nu_+$, then $X_{\alpha,\beta,k} = \emptyset$.
    On the other hand, if $\nu^k_-\subseteq\nu_+$, then 
    \[
        X_{\alpha, \beta, k}
        = \bigsqcup_{\nu \in [\nu^k_-, \nu_+]}
            \bigsqcup_{\substack{\ell(\lambda) \leq k \\ |\lambda| = |\alpha| - |\nu|}}
            X_{\alpha, \beta}(\lambda, \nu),
    \]
    where $[\nu^k_-, \nu_+]$ denotes the interval in Young's lattice bounded below by $\nu^k_-$ and above by $\nu_+$.
\end{prop}
\begin{proof}
    Note that $\alpha / \nu$ and $\beta / \nu'$ are proper skew shapes only if $\nu \subseteq \alpha$ and $\nu' \subseteq \beta$, respectively.
    Therefore, $X_{\alpha, \beta}(\lambda, \nu)$ is nonempty only if $\nu \subseteq \nu_+$.
    
    Furthermore, the columns of Littlewood--Richardson tableaux are strictly increasing, and the constraint that $\ell(\lambda) \leq k$ means that tableaux of content $\lambda$ only contain labels less than or equal to $k$.
    Hence $X_{\alpha, \beta}(\lambda, \nu)$ is nonempty only if the shapes $\alpha / \nu$ and $\beta / \nu'$ each have at most $k$ boxes in each column, i.e., only if $\nu \supseteq \nu^k_-$.

    Finally, a tableau of shape $\alpha / \nu$ can have content $\lambda$ only if $|\alpha / \nu| = |\lambda|$, so $X_{\alpha, \beta}(\lambda, \nu)$ is nonempty only if $|\lambda| = |\alpha| - |\nu|$.
\end{proof}
We have shown that when $\nu^k_- \subseteq \nu_+$, the only partitions $\nu$ such that $X_{\alpha,\beta}(\lambda,\nu)$ is nonempty lie within the lattice interval $[\nu^k_-,\nu_+]$.
When $k = 1$, the interval $[\nu_-^1, \nu_+]$ is in fact trivial if and only if there exists a partition $\mu$ such that $\alpha = \mu_{(1)}$ and $\beta = (\mu')_{(1)}$, as we show in Propositions \ref{prop:nu_- = nu_+ if stab.} and \ref{prop:stab. if nu_- = nu_+}.
\begin{prop}
    \label{prop:nu_- = nu_+ if stab.}
    If $\alpha = \mu_{(k)}$ and $\beta = (\mu')_{(k)}$ for some partition $\mu$, then $\nu^k_- = \mu = \nu_+$.
\end{prop}
\begin{proof}
    By the definition of $k$-stabilization, $\alpha$ is the partition obtained by adding $k$ boxes to the bottom of the first $r(\mu)$ columns of $\mu$ and $\beta' = ((\mu')_{(k)})'$ is the partition obtained by adding $k$ boxes to the right of the first $r(\mu)$ rows of $\mu$.
    \begin{figure}[h]
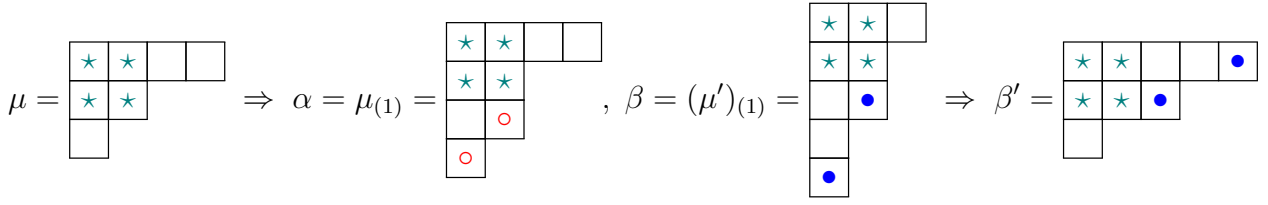

        \[
            \newcommand{\dmark}{\green{\star}}
            \newcommand{\amark}{\red{\circ}}
            \newcommand{\bmark}{\blue{\bullet}}
            \mu = \begin{yctableau}
                \dmark & \dmark & \ & \ \\
                \dmark & \dmark \\
                \
            \end{yctableau}
            ~~\Rightarrow~~
            \alpha = \mu_{(1)} = \begin{yctableau}
                \dmark & \dmark & \ & \ \\
                \dmark & \dmark \\
                \ & \amark \\
                \amark
            \end{yctableau},
            ~\beta = (\mu')_{(1)} = \begin{yctableau}
                \dmark & \dmark & \ \\
                \dmark & \dmark \\
                \ & \bmark \\
                \ \\
                \bmark
            \end{yctableau}
            ~~\Rightarrow~~ \beta' = \begin{yctableau}
                \dmark & \dmark & \ & \ & \bmark \\
                \dmark & \dmark & \bmark \\
                \
            \end{yctableau}
        \]
        \caption{The construction of $\alpha$ and $\beta$ from $\mu$}
        \label{fig:alpha/beta construction}
    \end{figure}
    
    Each box in $\alpha / \mu$ (marked with red circles in Figure \ref{fig:alpha/beta construction}) is necessarily below row $r(\mu)$ and within the first $r(\mu)$ columns of $\alpha$.
    Likewise, each box in $\beta' / \mu$ (marked with blue bullets) is necessarily right of the column $r(\mu)$ and within the first $r(\mu)$ rows.
    Therefore, $\alpha / \mu$ and $\beta' / \mu$ are disjoint.
    Hence $\nu_+ = \alpha \cap \beta' = \mu$.
    
    Furthermore, the first $r(\mu)$ columns of $\ctrim_k(\alpha)$ coincide with the first $r(\mu)$ columns of $\mu$ (since $\alpha / \ctrim_k(\alpha)$ and $\alpha / \mu$ coincide in those columns), and the latter columns of $\ctrim_k(\alpha)$ are strictly contained in the latter columns of $\mu$.
    Likewise, the first $r(\mu)$ rows of $\rtrim_k(\beta')$ coincide with the first $r(\mu)$ rows of $\mu$, and the latter rows of $\rtrim_k(\beta')$ are strictly contained in the latter rows of $\mu$.
    By the definition of rank, each box in $\mu$ is either in the first $r(\mu)$ rows or the first $r(\mu)$ columns; therefore, $\nu^k_- = \ctrim_k(\alpha) \cup \rtrim_k(\beta') = \mu$.
\end{proof}
\begin{theorem}
    \label{thm:dual coeff. = 1 if stab.}
    If $\alpha = \mu_{(k)}$ and $\beta = (\mu')_{(k)}$ for some partition $\mu$, then $X_{\alpha, \beta, k}$ contains exactly one element, which has sign $(-1)^{|\mu|}$.
\end{theorem}
\begin{proof}
    By Proposition \ref{prop:nu_- = nu_+ if stab.}, $\nu^k_- = \mu = \nu_+$, and so by Proposition \ref{prop:nu interval},
    \[
        X_{\alpha, \beta, k}
        = \bigsqcup_{\substack{\ell(\lambda) \leq k \\ |\lambda| = |\alpha| - |\mu|}}
            X_{\alpha, \beta}(\lambda, \mu).
    \]
    Note that since $\alpha = \mu_{(k)}$, the skew shape $\alpha / \mu$ has $r(\mu)$ columns with $k$ boxes in each column.
    Filling the $i$-th box in each column with $i$ gives the unique Littlewood--Richardson tableau with shape $\alpha/\mu$ and entries of size at most $k$.
    Likewise, there is exactly one Littlewood--Richardson tableau with shape $\beta/\mu'$ and entries of size at most $k$.
\end{proof}

\subsection{The involution case}
This subsection proves Theorem~\ref{thm:k = 1 solution}, solving the $k=1$ case of Problem~\ref{prob:dual}(2). Let $\alpha, \beta$ be a pair of partitions such that $|\alpha| = |\beta|$.
Let $\nu_+ = \nu_+(\alpha, \beta')$ and let $\nu_-^1 = \nu_-^1(\alpha, \beta')$.
\begin{lemma}
    \label{lem:dual coeffs. k = 1}
    Let $\nu \in [\nu^1_-, \nu_+]$.
    Then there exists a bijection between the interval $[\nu^1_-, \nu_+]$ and the set $X_{\alpha, \beta, 1}$.
\end{lemma}
\begin{proof}
    Let $\lambda$ be a row partition.
    By Proposition \ref{prop:nu interval}, $C_{\alpha, \beta}(\lambda, \nu) \neq 0$ only if $|\lambda| = |\alpha| - |\nu|$.
    Furthermore, $\nu \supseteq \nu^1_- = \ctrim_1(\alpha) \cup \rtrim_1(\beta')$.
    Hence $\ctrim_1(\alpha) \subseteq \nu$, so $\alpha / \nu \subseteq \alpha / \ctrim_1(\alpha)$.
    And by construction, $\alpha / \ctrim_1(\alpha)$ is a horizontal strip, so $\alpha / \nu$ must be a horizontal strip as well.
    Therefore, if $|\lambda| = |\alpha| - |\nu|$, by the Pieri rule (Theorem \ref{thm:Pieri rule}), $c^\alpha_{\nu, \lambda} = 1$.
    Likewise, $\beta' / \rtrim_1(\beta')$ is a vertical strip, so $\beta' / \nu$ must be a vertical strip as well.
    Therefore, if $|\lambda| = |\beta| - |\nu|$, by the dual Pieri rule (Theorem \ref{thm:dual Pieri rule}), $c^{\beta'}_{\nu, \lambda} = 1$.

    For all $\nu$, we have shown that if $\lambda$ is the row partition $(|\alpha| - |\nu|)$, then $C_{\alpha, \beta}(\lambda, \nu) = 1$, and otherwise, $C_{\alpha, \beta}(\lambda, \nu) = 0$.
    Mapping each $\nu$ to the unique element of $X_{\alpha,\beta}((|\alpha|-|\nu|), \nu)$ thus gives the desired bijection.
\end{proof}
\begin{prop}
    \label{prop:stab. if nu_- = nu_+}
    If $\nu^1_- = \nu_+$, then $\alpha = (\nu_+)_{(1)}$ and $\beta = (\nu_+')_{(1)}$.
\end{prop}
\begin{proof}
    Let $C$, $u$, and $v$ be as in Lemma \ref{lem:stab. shapes if nu_- = nu_+} (taking $\lambda = \alpha$ and $\mu = \beta'$).
    Then by Lemma \ref{lem:stab. shapes if nu_- = nu_+}, $\alpha / \nu_+$ consists of exactly one box in each of the first $j_u$ columns and $\beta' / \nu_+$ consists of exactly one box in each of the first $i_v$ rows.
    Hence $|\alpha / \nu_+| = j_u$ and $|\beta / \nu_+| = i_v$.
    But by assumption, $|\alpha| = |\beta|$; therefore, $j_u = i_v$.
    
    Note that $\mathsf c_u$ and $\mathsf c_v$ are boxes in column $j_u$ and row $i_v$ of $\nu_+$, respectively; therefore, $r(\nu_+) \geq j_u = i_v$.
    But by Lemma \ref{lem:* if nu_- = nu_+}(3), either $u = v$ or $u = v-1$, that is, the corners $\mathsf c_u$ and $\mathsf c_v$ are either identical or consecutive.
    Therefore, the box $(i_v + 1, j_u + 1)$ is not contained in $\nu_+$, and so we conclude that $r(\nu_+) = j_u$.

    To restate, we have found that $\alpha / \nu_+$ consists of exactly one box in each of the first $r(\nu_+)$ columns.
    In other words, $\alpha$ is obtained by adding a row of length $r(\nu_+)$ to $\nu_+$, that is, $\alpha = (\nu_+)_{(1)}$.
    Likewise, $\beta' / \nu_+$ consists of exactly one box in each of the first $r(\nu_+)$ rows.
    Therefore, $\beta$ is obtained by adding a row of length $r(\nu_+)$ to $\nu_+'$, that is, $\beta = (\nu_+')_{(1)}$.
\end{proof}
\begin{theorem}
    \label{thm:k = 1 solution}
    Suppose that there does not exist a partition $\mu$ such that $\alpha = \mu_{(1)}$ and $\beta = (\mu')_{(1)}$.
    Then there exists a sign-reversing involution on $X_{\alpha, \beta, 1}$.
\end{theorem}
\begin{proof}
    Suppose that $\nu^1_- \not\subseteq \nu_+$.
    Then by Proposition \ref{prop:nu interval}, $X_{\alpha, \beta, 1}$ is empty, and so there exists a trivial involution.

    Now suppose that $\nu^1_- \subseteq \nu_+$.
    Then by the contrapositive of Proposition \ref{prop:stab. if nu_- = nu_+}, $\nu^1_- \neq \nu_+$.
    So we can choose a box $\mathsf c$ from the skew shape $\nu_+ / \nu^1_-$.
    By Lemma \ref{lem:nu corners}, every box in $\nu_+ / \nu^1_-$ is a corner box, so adding or removing $\mathsf c$ from a partition $\nu \in [\nu^1_-, \nu_+]$ will always yield another valid partition.
    Therefore, the map $\toggle_{\mathsf c} \colon [\nu^1_-, \nu_+] \leftrightarrow [\nu^1_-, \nu_+]$ which either adds or removes $\mathsf c$ is an involution.

    Furthermore, by Lemma \ref{lem:dual coeffs. k = 1}, there exists a bijection $\phi \colon [\nu^1_-, \nu_+] \to X_{\alpha, \beta, 1}$.
    Since $\nu$ and $\toggle_{\mathsf c}(\nu)$ differ by exactly one box, $\phi(\nu)$ and $\phi(\toggle_{\mathsf c}(\nu))$ are of opposite sign.
    Therefore, we can construct our desired sign-reversing involution by mapping $\phi(\nu)$ to $\phi(\toggle_{\mathsf c}(\nu))$.
\end{proof}

\section{Proof of the \texorpdfstring{$k=\min\{m, n\}-1$}{submaximal} case}
\label{sec:classical Cauchy}
The goal of this section is to solve the $k=\min\{m, n\}-1$ case of Problem~\ref{prob:main}. 
Mimicking the strategy of Section~\ref{sec:dual Cauchy}, we begin by using the classical Cauchy identity \eqref{eqn:classical Cauchy} to derive Problem~\ref{prob:classical}, another combinatorial reformulation of the generalized Cauchy identity \eqref{eqn:gen. Cauchy}.

\begin{definition}
    Let $\alpha, \beta$ be a pair of partitions such that $|\alpha| = |\beta|$.
    Then we define
    \[
        X'_{\alpha, \beta, k}(\mu, \nu)
        = \LR(\alpha / \nu, \mu_{(k)}) \times \LR(\beta / \nu, (\mu')_{(k)}),
        \quad
        X'_{\alpha, \beta, k}
        = \bigsqcup_{\substack{\mu, \nu \\ \mu \subseteq [m - k] \times [n - k]}} X'_{\alpha, \beta, k}(\mu, \nu).
    \]
    Let the \emph{sign} of $\chi \in X_{\alpha, \beta, k}(\mu, \nu)$ be $\sgn(\chi) = (-1)^{|\mu|}$.
    Then we define
    \begin{equation}
        \begin{split}
            C'_{\alpha, \beta, k}(\mu, \nu)
            &= |X'_{\alpha, \beta, k}(\mu, \nu)|
            = c^\alpha_{\nu, \mu_{(k)}}c^\beta_{\nu, (\mu')_{(k)}}, \\
            C'_{\alpha, \beta, k}
            &= \sum_{\chi \in X'_{\alpha, \beta, k}} \sgn(\chi)
            = \sum_{\substack{\mu, \nu \\ \mu \subseteq [m - k] \times [n - k]}} (-1)^{|\mu|} C_{\alpha, \beta, k}(\mu, \nu).
        \end{split}
    \end{equation}
\end{definition}
\begin{theorem}
    \label{thm:classical Cauchy drv.}
    The generalized Cauchy identity \eqref{eqn:gen. Cauchy} is equivalent to the statement that
    \begin{equation}
        \label{eqn:classical coeffs.}
        C'_{\alpha, \beta, k} = \begin{cases}
            1 & \text{if $\alpha = \beta$ and $\ell(\alpha) \leq k$}, \\
            0 & \text{otherwise}.
        \end{cases}
    \end{equation}
\end{theorem}
\begin{proof}
    Apply the classical Cauchy identity \eqref{eqn:classical Cauchy} to rewrite the right side of \eqref{eqn:gen. Cauchy}:
    \[
        \sum_{\ell(\lambda) \leq k} s_\lambda(\vec x)s_\lambda(\vec y)
        = \bigg[\sum_\nu s_\nu(\vec x)s_\nu(\vec y)\bigg]\bigg[
            \sum_{\mu \subseteq [m - k] \times [n - k]} (-1)^{|\mu|}
            s_{\mu_{(k)}}(\vec x)s_{(\mu')_{(k)}}(\vec y)
        \bigg].
    \]
    Then apply the Littlewood--Richardson rule (\ref{eqn:LR rule}) twice on the right-hand side:
    \[
        \sum_{\ell(\lambda) \leq k} s_\lambda(\vec x)s_\lambda(\vec y)
        = \sum_{\alpha, \beta} \Bigg[
            \bigg(
                \sum_{\substack{\mu, \nu \\ \mu \subseteq [m - k] \times [n - k]}}
                (-1)^{|\mu|}c^\alpha_{\nu, \mu_{(k)}}c^\beta_{\nu, (\mu')_{(k)}}
            \bigg)
            s_\alpha(\vec x)s_\beta(\vec y)
        \Bigg].
    \]
    Note that, by Remark \ref{rmk:Schur vanishing}, $s_\alpha(\vec x) = 0$ if $\ell(\alpha) > m$ and $s_\beta(\vec y) = 0$ if $\ell(\beta) > n$.

    As in the proof of Theorem \ref{thm:dual Cauchy drv.}, the Schur polynomials form a linear basis of the space of symmetric functions, so the coefficients of any two decompositions must be equal.
\end{proof}
By Theorem \ref{thm:classical Cauchy drv.}, to solve Problem \ref{prob:main} it suffices to solve the following problem:
\begin{problem}
    \label{prob:classical}
    Let $\alpha, \beta$ be a pair of partitions such that $|\alpha|=|\beta|$ and let $k \in \Z_{\geq 0}$.
    \begin{enumerate}
    \item
        Prove that if $\alpha = \beta$ and $\ell(\alpha) \leq k$, then $X'_{\alpha, \beta, k}$ contains exactly one element, which has sign $1$.
    \item Otherwise, construct a sign-reversing involution on $X'_{\alpha, \beta, k}$.
    \end{enumerate}
\end{problem}
The following subsections solve Problem~\ref{prob:classical}(1) for all $k$ and Problem~\ref{prob:classical}(2) when $k=\min\{m, n\}-1$.

\subsection{The non-involution case}
Problem~\ref{prob:classical}(1) turns out to be simpler than Problem~\ref{prob:dual}(1); it is solved (for all $k$) by the following proposition.
\begin{prop}
    \label{prop:classical coeffs. = 1 if equal}
    Let $\alpha$ be a partition such that $\ell(\alpha) \leq k$.
    Then
    \[
        C'_{\alpha, \alpha, k}(\mu, \nu) = \begin{cases}
            1 & \text{if $\mu = \emptyset$ and $\alpha = \nu$,} \\
            0 & \text{otherwise.}
        \end{cases}
    \]
\end{prop}
\begin{proof}
    Suppose that $\mu = \emptyset$.
    Then $\mu_{(k)} = (\mu')_{(k)} = \emptyset$, and the trivial tableau of weight $1$ has shape $\alpha / \alpha$.

    Now suppose that $\mu \neq \emptyset$.
    Then $\mu_{(k)}$ and $(\mu')_{(k)}$ both have at least $k + 1$ rows.
    Therefore, if $T$ is a Littlewood--Richardson tableau of content $\mu_{(k)}$, each of the labels $1, \ldots, k + 1$ appear at least once in $T$.
    And by Lemma \ref{lem:LR distinct rows}, each of these labels must first occur in distinct rows.
    But by assumption, $\ell(\alpha) \leq k$, so $T$ cannot have shape $\alpha / \nu$.
\end{proof}

\subsection{The involution case}
This subsection proves Theorem~\ref{thm:classical involution}, solving the $k = \min\{m, n\} - 1$ case of Problem~\ref{prob:classical}(2).
Assume without loss of generality that $m \leq n$, so that $k = m - 1$, and let $\alpha, \beta$ be a pair of partitions such that $|\alpha| = |\beta|$, $\ell(\alpha) \leq m$, and $\ell(\beta) \leq n$.
Let $\nu_+ = \nu_+(\rtrim_1(\alpha), \beta)$ and let $\nu^1_- = \nu_-^1(\rtrim_1(\alpha), \beta)$.

The next proposition effectively solves one case of Problem~\ref{prob:classical}(2) when $k=m-1$.
\begin{prop}
    \label{prop:k = m - 1 exception}
    Let $\mu$ be a row partition.
    If $\ell(\alpha) = m$, then
    \[
        C'_{\alpha, \alpha, m - 1}(\mu, \nu) = \begin{cases}
            1 & \text{if $\mu = \emptyset$ and $\alpha = \nu$}, \\
            1 & \text{if $\mu = (1)$ and $\nu = \rtrim_1(\alpha) = \nu_+$}, \\
            0 & \text{otherwise}.
        \end{cases}
    \]
\end{prop}
\begin{proof}
    First, suppose that $|\mu| = 0$.
    Then $\mu_{(m-1)} = (\mu')_{(m-1)} = \emptyset$, and the trivial tableau of weight $1$ has shape $\alpha / \alpha$.

    Now suppose that $|\mu| = 1$.
    Then the $(m-1)$-stabilizations $\mu_{(m-1)} = (\mu')_{(m-1)} = (1^m)$ are both equal to the column partition of size $m$.
    Thus, by the dual Pieri rule (Theorem \ref{thm:dual Pieri rule}),
    \[
        c^\alpha_{\nu, \mu_{(m-1)}} = c^\alpha_{\nu, (\mu')_{(m-1)}} = \begin{cases}
            1 & \text{if $\alpha / \nu$ is a vertical strip of size $m$}, \\
            0 & \text{otherwise}.
        \end{cases}
    \]
    But $\alpha / \nu$ is a vertical strip of size $m$ if and only if $\nu = \rtrim_1(\alpha)$.

    Finally, suppose that $\mu$ is a row partition with $|\mu|>1$.
    Then $(\mu')_{(m-1)}$ is a column partition with at least $m + 1$ boxes.
    But $\ell(\alpha) = m$, so $\alpha / \nu$ cannot be a vertical strip with at least $m+1$ boxes for any partition $\nu$.
    Thus $c^\alpha_{\nu, (\mu')_{(m-1)}} = 0$ for all $\nu$ by the dual Pieri rule.
\end{proof}
\begin{prop}
    \label{prop:k = m - 1 nu interval}
    Suppose that $\alpha \neq \beta$.
    If $\nu^1_- \subseteq \nu_+$ and $\ell(\alpha) = m$, then there exists a bijection between the interval $[\nu_-^1, \nu_+]$ and the set $X'_{\alpha, \beta, m - 1}$.
    On the other hand, if $\nu^1_- \not\subseteq \nu_+$ or $\ell(\alpha) < m$, then $X'_{\alpha, \beta, m - 1}$ is empty.
\end{prop}
\begin{proof}
    We have that
    \[
        X'_{\alpha, \beta, m - 1}
        = \bigsqcup_{\substack{\mu, \nu \\ \mu \subseteq [m - (m - 1)] \times [n - (m - 1)]}} X'_{\alpha, \beta, m - 1}(\mu, \nu)
        = \bigsqcup_{\substack{\mu, \nu \\ \mu \subseteq [1] \times [n - m + 1]}} X'_{\alpha, \beta, m - 1}(\mu, \nu).
    \]
    Note that $[1] \times [n - m + 1] = (n - m + 1)$ is a row partition, so those $\mu$ appearing in the above equation are row partitions as well.
    Furthermore, a tableau of shape $\alpha / \nu$ can have content $\mu_{(k)}$ only if $|\alpha / \nu| = |\mu_{(k)}|$, so $X'_{\alpha, \beta, m - 1}(\mu, \nu)$ is nonempty only if $\mu = (|\alpha| - |\nu| - m+1)$.
    Hence if $X'_{\alpha, \beta, m - 1}(\mu, \nu)$ is nonempty, then $\mu_{(m-1)}$ is the hook partition $H(|\mu|, m)$ and $(\mu')_{(m-1)}$ is the column partition $(1^{|\mu| + m - 1}) = (1^{|\alpha|-|\nu|})$.

    Therefore, by the hook Pieri rule (Theorem \ref{thm:hook Pieri rule}),
    \[
        c^\alpha_{\nu, H(|\mu|, m)} = \begin{cases}
            1 & \text{if $\ell(\alpha) = m$, $\nu \subseteq \rtrim_1(\alpha)$, and $\rtrim_1(\alpha) / \nu$ is a horizontal strip}, \\
            0 & \text{otherwise}.
        \end{cases}
    \]
    Likewise, by the dual Pieri rule,
    \[
        c^\beta_{\nu, (1^{|\alpha| - |\nu|})} = \begin{cases}
            1 & \text{if $\nu \subseteq \beta$ and $\beta / \nu$ is a vertical strip}, \\
            0 & \text{otherwise}.
        \end{cases}
    \]
    Note that $\nu \subseteq \rtrim_1(\alpha)$ and $\nu \subseteq \beta$ if and only if $\nu \subseteq \nu_+$.
    Similarly, $\rtrim_1(\alpha) / \nu$ is a horizontal strip and $\beta / \nu$ is a vertical strip if and only if $\nu \supseteq \nu^1_-$.
    Therefore,
    \[
        C'_{\alpha, \beta, m-1}(\mu, \nu)
        = c^\alpha_{\nu, \mu_{(m-1)}}c^\beta_{\nu, (\mu')_{(m-1)}} = \begin{cases}
            1 & \text{if $\ell(\alpha) = m$ and $\nu \in [\nu^1_-, \nu_+]$}, \\
            0 & \text{otherwise}.
        \end{cases}
    \]
    For all $\nu \in [\nu^1_-, \nu_+]$, we have shown that if $\mu$ is the row partition $(|\alpha| - |\nu| - m+1)$, then $C'_{\alpha, \beta, m-1}(\mu, \nu) = 1$, and otherwise, $C'_{\alpha, \beta, m-1}(\mu, \nu) = 0$.
    Associating each $\nu$ to the unique element of $X'_{\alpha, \beta, m-1}((|\alpha|-|\nu|-m+1), \nu)$ gives us our desired bijection.
\end{proof}
\begin{prop}
    \label{prop:equal if nu_- = nu_+}
    If $\nu^1_- = \nu_+$, then $\alpha = \beta$.
\end{prop}
\begin{proof}
    Let $C$, $u$, and $v$ be as in Lemma~\ref{lem:stab. shapes if nu_- = nu_+} (taking $\lambda = \rtrim_1(\alpha)$ and $\mu = \beta$).
    Then by Lemma \ref{lem:stab. shapes if nu_- = nu_+}, $\rtrim_1(\alpha) / \nu_+$ consists of one box in each of the first $j_u$ columns and $\beta/\nu_+$ consists of one box in each of the first $i_v$ rows.
    Hence $|\rtrim_1(\alpha) / \nu_+| = j_u$ and $|\beta / \nu_+| = i_v$.
    But by assumption, $|\alpha| = |\beta|$, and so
    \[
        j_u
        = |\rtrim_1(\alpha) / \nu_+|
        = |\alpha / \nu_+| - |\alpha / \rtrim_1(\alpha)|
        = |\beta / \nu_+| - \ell(\alpha)
        = i_v - \ell(\alpha).
    \]
    Furthermore, $\mathsf c_v \in \rtrim_1(\alpha)\subseteq\alpha$, so $i_v \leq \ell(\alpha)$.
    Therefore, $i_v = \ell(\alpha)$ and $j_u = |\rtrim_1(\alpha) / \nu_+| = 0$, that is, $\nu_+ = \rtrim_1(\alpha)$.

    We have that $\beta / \nu_+ = \beta / \rtrim_1(\alpha)$ consists of exactly one box in each of the first $\ell(\alpha)$ rows.
    But by construction, $\alpha / \rtrim_1(\alpha)$ also consists of exactly one box in each of the first $\ell(\alpha)$ rows.
    Therefore, $\alpha = \beta$.
\end{proof}
\begin{theorem}
    \label{thm:classical involution}
    Suppose either that $\alpha \neq \beta$ or that $\alpha = \beta$ and $\ell(\alpha) = m$.
    Then there exists a sign-reversing involution on $X'_{\alpha, \beta, m - 1}$.
\end{theorem}
\begin{proof}
    There are three cases to consider.
    \begin{enumerate}
    \item
        Suppose that $\alpha = \beta$ and $\ell(\alpha) = m$.
        Then by Proposition \ref{prop:k = m - 1 exception}, $X'_{\alpha, \beta, m - 1}$ contains two elements, one corresponding to $\mu = \emptyset$ and one corresponding to $\mu = (1)$.
        These two elements are of opposite sign, so mapping them to each other gives the desired involution.
    \item
        Suppose that $\alpha \neq \beta$, but $\nu^1_- \not\subseteq \nu_+$ or $\ell(\alpha) < m$.
        Then by Proposition \ref{prop:k = m - 1 nu interval}, $X'_{\alpha, \beta, m - 1}$ is empty, and so there exists a trivial involution.
    \item
        Suppose that $\alpha \neq \beta$, $\nu^1_- \subseteq \nu_+$, and $\ell(\alpha) = m$.
        Then by the contrapositive of Proposition~\ref{prop:equal if nu_- = nu_+}, $\nu^1_- \neq \nu_+$.
        So we can choose a box $\mathsf c$ from the skew shape $\nu_+ / \nu^1_-$.
        By Lemma~\ref{lem:nu corners}, every box in $\nu_+ / \nu^1_-$ is a corner box, so adding or removing $\mathsf c$ from a partition $\nu \in [\nu^1_-, \nu_+]$ will always yield a valid partition.
        Therefore, the map $\toggle_{\mathsf c} \colon [\nu^1_-, \nu_+] \leftrightarrow [\nu^1_-, \nu_+]$ which either adds or removes $\mathsf c$ is an involution.

        Furthermore, by Proposition~\ref{prop:k = m - 1 nu interval}, there exists a bijection $\phi \colon [\nu^1_-, \nu_+] \to X'_{\alpha, \beta, m-1}$.
        Since $\nu$ and $\toggle_{\mathsf c}(\nu)$ differ by exactly one box, $\phi(\nu)$ and $\phi(\toggle_{\mathsf c}(\nu))$ are of opposite sign.
        Therefore, we can construct our desired sign-reversing involution by mapping $\phi(\nu)$ to $\phi(\toggle_{\mathsf c}(\nu))$.
    \end{enumerate}
\end{proof}

\section{Further directions}\label{sec:further}
In this final section, we discuss challenges in extending our proof methods to solve Problem~\ref{prob:main} for all $k$. 
Section~\ref{sec:dual Cauchy} reduces Problem~\ref{prob:main} to the problem of constructing an explicit sign-reversing involution on the set $X_{\alpha,\beta,k}$ for each $k$ and pair of partitions $(\alpha,\beta)$ with $|\alpha| = |\beta|$ (Problem~\ref{prob:dual}(2)). 
Following the notation of Section~\ref{sec:dual Cauchy}, let
\[Y_{\alpha, \beta, k}(\nu) = \bigsqcup_{\substack{\ell(\lambda) \leq k \\ |\lambda| = |\alpha| - |\nu|}} X_{\alpha, \beta}(\lambda, \nu) \quad \text{ and }\quad  Y_{\alpha,\beta, k}(t) = \bigsqcup_{|\nu| = t} Y_{\alpha,\beta,k}(\nu).\]
By Proposition~\ref{prop:nu interval}, we then have
\[X_{\alpha,\beta,k} = \bigsqcup_{\nu\in[\nu_-^k,\nu_+]}Y_{\alpha,\beta,k}(\nu) = \bigsqcup_{t = |\nu_-^k|}^{|\nu_+|} Y_{\alpha,\beta,k}(t).\]
Lemma~\ref{lem:dual coeffs. k = 1} shows that when $k=1$, each set $Y_{\alpha,\beta,1}(\nu)$ has a unique element, while Lemma~\ref{lem:nu corners} implies that the interval $[\nu_-^1, \nu_+]$ is a subset lattice on the boxes of $\nu_+/\nu_-^1$. 
Together, these results enable the construction of a simple sign-reversing involution in the proof of Theorem~\ref{thm:k = 1 solution}, solving the $k=1$ case of Problem~\ref{prob:main}. 

There are two obstacles in extending this proof strategy for $k>1$. 
First, for $k>1$ the sets $Y_{\alpha,\beta,k}(\nu)$ may have multiple elements or no elements at all.
Second, for $k>1$ the interval $[\nu_-^k, \nu_+]$ need not be a subset lattice. 
Thus the sequence of cardinalities $(|Y_{\alpha,\beta, k}(t)|)_{t=|\nu_-^k|}^{|\nu_+|}$ is not generally a binomial sequence for larger values of $k$---indeed, this sequence may fail to be symmetric or unimodal, and may have internal zeroes.
The following examples illustrate these phenomena.

\begin{example}
    \label{ex:symmetric failure}
    Let $\alpha = (2, 2, 1)$, let $\beta = (3, 1, 1)$, and let $k = 2$.
    Then the elements of $X_{\alpha, \beta, 2}$ are as follows (boxes not in the skew shapes are grayed out):
    \[
        \left(
            ~\begin{yctableau}
                *(gray) & *(gray) \\
                *(gray) & 1 \\
                *(gray)
            \end{yctableau}
            ~,\quad\begin{yctableau}
                *(gray) & *(gray) & *(gray) \\
                *(gray) \\
                1
            \end{yctableau}~
        \right),
        \left(
            ~\begin{yctableau}
                *(gray) & *(gray) \\
                *(gray) & 1 \\
                1
            \end{yctableau}
            ~,\quad\begin{yctableau}
                *(gray) & *(gray) & 1 \\
                *(gray) \\
                1
            \end{yctableau}~
        \right),
        \left(
            ~\begin{yctableau}
                *(gray) & *(gray) \\
                *(gray) & 1 \\
                2
            \end{yctableau}
            ~,\quad\begin{yctableau}
                *(gray) & *(gray) & 1 \\
                *(gray) \\
                2
            \end{yctableau}~
        \right),
    \]
    \[
        \left(
            ~\begin{yctableau}
                *(gray) & 1 \\
                *(gray) & 2 \\
                *(gray)
            \end{yctableau}
            ~,\quad\begin{yctableau}
                *(gray) & *(gray) & *(gray) \\
                1 \\
                2
            \end{yctableau}~
        \right),
        \left(
            ~\begin{yctableau}
                *(gray) & 1 \\
                *(gray) & 2 \\
                1
            \end{yctableau}
            ~,\quad\begin{yctableau}
                *(gray) & *(gray) & 1 \\
                1 \\
                2
            \end{yctableau}~
        \right),
        \left(
            ~\begin{yctableau}
                *(gray) & *(gray) \\
                1 & 1 \\
                2
            \end{yctableau}
            ~,\quad\begin{yctableau}
                *(gray) & 1 & 1 \\
                *(gray) \\
                2
            \end{yctableau}~
        \right).
    \]

    Grouping the elements by the size $t$ of the inner shape $\nu$ gives the following sequence:
    \begin{center}
        \begin{tabular}{c|cccccc}
             $t$ & 0 & 1 & 2 & 3 & 4 & 5 \\
             \hline
             $C_{\alpha, \beta, 2}(t)$ & 0 & 0 & 2 & 3 & 1 & 0
        \end{tabular}
    \end{center}
    This sequence is not symmetric.
\end{example}
\begin{example}
    \label{ex:unimodal failure}
    Let $\alpha = (3, 3)$, let $\beta = (2, 2, 2)$, and let $k = 2$.
    Then the elements of $X_{\alpha, \beta, 2}$ are as follows:
    \[
        \left(
            ~\begin{yctableau}
                *(gray) & *(gray) & *(gray) \\
                *(gray) & *(gray) & *(gray) \\
            \end{yctableau}
            ~,\quad\begin{yctableau}
                *(gray) & *(gray) \\
                *(gray) & *(gray) \\
                *(gray) & *(gray)
            \end{yctableau}~
        \right),
        \left(
            ~\begin{yctableau}
                *(gray) & *(gray) & *(gray) \\
                *(gray) & *(gray) & 1 \\
            \end{yctableau}
            ~,\quad\begin{yctableau}
                *(gray) & *(gray) \\
                *(gray) & *(gray) \\
                *(gray) & 1
            \end{yctableau}~
        \right),
    \]
    \[
        \left(
            ~\begin{yctableau}
                *(gray) & *(gray) & 1 \\
                *(gray) & 1 & 2 \\
            \end{yctableau}
            ~,\quad\begin{yctableau}
                *(gray) & *(gray) \\
                *(gray) & 1 \\
                1 & 2
            \end{yctableau}~
        \right),
        \left(
            ~\begin{yctableau}
                *(gray) & 1 & 1 \\
                *(gray) & 2 & 2 \\
            \end{yctableau}
            ~,\quad\begin{yctableau}
                *(gray) & *(gray) \\
                1 & 1 \\
                2 & 2
            \end{yctableau}~
        \right).
    \]
    Grouping the elements by the size $t$ of the inner shape $\nu$ gives the following sequence:
    \begin{center}
        \begin{tabular}{c|ccccccc}
             $t$ & 0 & 1 & 2 & 3 & 4 & 5 & 6\\
             \hline
             $C_{\alpha, \beta, 2}(t)$ & 0 & 0 & 1 & 1 & 0 & 1 & 1
        \end{tabular}
    \end{center}
    This sequence is not unimodal and contains internal zeroes.
\end{example}

\begin{table}[ht]
    \begin{tabular}{c|c}
        Sequence & Generating $\alpha$ \\
        \hline
        $1$ & $(0)$ \\
        $1, 1$ & $(1)$ \\
        $1, 2, 2, 1$ & $(2, 1)$ \\
        $1, 1, 0, 1, 1$ & $(2, 2)$ \\
        $1, 2, 2, 2, 1$ & $(2, 2, 1)$ \\
        $1, 3, 6, 7, 4, 1$ & $(3, 2, 1)$ \\
        $1, 3, 6, 8, 5, 1$ & $(4, 2, 1)$ \\
        $1, 2, 2, 3, 3, 1$ & $(3, 3, 1)$ \\
        $1, 3, 6, 8, 6, 2$ & $(4, 3, 1)$ \\
        $1, 2, 2, 4, 5, 2$ & $(4, 2, 2)$ \\
        $1, 3, 6, 9, 7, 2$ & $(4, 2, 1, 1)$ \\
        $1, 2, 2, 4, 5, 2$ & $(3, 3, 1, 1)$
    \end{tabular}
    \caption{Table of sequences $(C_{\alpha,\alpha',2}(t))_{t\in\Z_{\geq 0}}$}\label{tab:seqs}
\end{table}

Table~\ref{tab:seqs} lists all of the sequences of the form $(C_{\alpha, \alpha', 2}(t))_{t \in \Z_{\geq 0}}$, where $|\alpha| \leq 8$. 
For general $\alpha, \beta$ of equal size, the sequence $(C_{\alpha, \beta, k}(t))_{t \in \Z_{\geq 0}}$ will consist of $|\alpha| - |\alpha \cap \beta'|$ many initial zeroes, followed by the sequence $(C_{\alpha \cap \beta', \alpha' \cap \beta, k}(t))_{t \in \Z_{\geq0}}$.

We observe that in all of these examples, there still exists a sign-reversing involution on $X_{\alpha,\alpha',2}$, which, like $\toggle$, adds or removes exactly one box from the inner shape.
For instance, in Example \ref{ex:symmetric failure}, we could map each pair to its counterpart in the same column.
Likewise, in Example \ref{ex:unimodal failure}, we could map each pair to its counterpart in the same row.
Despite these small examples, we do not yet know how to construct such an involution in general.

\section*{Acknowledgements}
We thank Jaewon Min, Abigail Price, Linus Setiabrata, and Alexander Yong for helpful conversations during the preparation of this material. 
This project began as part of the ICLUE program at UIUC in Summer 2025, organized by Alexander Yong and supported by an NSF RTG in Combinatorics (DMS 1937241). 
AS was also supported by an NSF graduate fellowship under Grant No. DGE 21-46756.

\end{document}